\documentclass[11pt]{article}

\usepackage[T1]{fontenc}
\usepackage[utf8]{inputenc}
\usepackage{lmodern}
\usepackage{amsthm,amsmath,amsfonts,amssymb,mathtools}
\usepackage{aliascnt}
\usepackage[numbers,sort&compress]{natbib}
\usepackage[margin=1in]{geometry}
\usepackage{enumitem}
\usepackage{xcolor}
\usepackage{microtype}
\usepackage[
  hypertexnames=false,
  colorlinks=true,
  linkcolor=blue,
  citecolor=blue,
  urlcolor=blue
]{hyperref}
\usepackage[nameinlink,capitalize]{cleveref}

\newcommand{\PaperPDFAuthors}{Tianle Liu}
\hypersetup{
  bookmarksdepth=3,
  pdftitle={Stein Kernels and Normal Approximation for Log-Concave Bilinear Forms},
  pdfauthor={\PaperPDFAuthors},
  pdfsubject={Probability; log-concave bilinear forms; Stein kernels},
  pdfkeywords={Jiang-Lee-Vempala conjecture, log-concave bilinear form, spectral participation ratio, effective rank, Stein kernel, Wasserstein distance, moment map}
}
\newcommand{\R}{\mathbb R}
\newcommand{\E}{\mathbb E}
\newcommand{\Law}{\mathcal L}
\newcommand{\Var}{\operatorname{Var}}
\newcommand{\Cov}{\operatorname{Cov}}
\newcommand{\Tr}{\operatorname{Tr}}
\newcommand{\HS}{\mathrm{HS}}
\newcommand{\dd}{\,\mathrm d}
\newcommand{\authorwithcontact}[3]{%
  #1\thanks{#2. Email: \href{mailto:#3}{\nolinkurl{#3}}.}%
}

\newtheoremstyle{thmstyleone}
  {6pt}{6pt}{\itshape}{}{}{.}{0.5em}{}
\newtheoremstyle{thmstyletwo}
  {6pt}{6pt}{\normalfont}{}{}{.}{0.5em}{}

\numberwithin{equation}{section}
\theoremstyle{thmstyleone}
\newtheorem{theorem}{Theorem}[section]
\newaliascnt{proposition}{theorem}
\newtheorem{proposition}[proposition]{Proposition}
\aliascntresetthe{proposition}
\newaliascnt{lemma}{theorem}
\newtheorem{lemma}[lemma]{Lemma}
\aliascntresetthe{lemma}
\newaliascnt{definition}{theorem}

\aliascntresetthe{definition}
\newaliascnt{corollary}{theorem}
\newtheorem{corollary}[corollary]{Corollary}
\aliascntresetthe{corollary}
\newaliascnt{conjecture}{theorem}

\aliascntresetthe{conjecture}
\theoremstyle{thmstyletwo}
\newaliascnt{example}{theorem}

\aliascntresetthe{example}
\newtheorem*{remark}{Remark}

\crefname{theorem}{theorem}{theorems}
\Crefname{theorem}{Theorem}{Theorems}
\crefname{proposition}{proposition}{propositions}
\Crefname{proposition}{Proposition}{Propositions}
\crefname{lemma}{lemma}{lemmas}
\Crefname{lemma}{Lemma}{Lemmas}
\crefname{definition}{definition}{definitions}
\Crefname{definition}{Definition}{Definitions}
\crefname{corollary}{corollary}{corollaries}
\Crefname{corollary}{Corollary}{Corollaries}
\crefname{conjecture}{conjecture}{conjectures}
\Crefname{conjecture}{Conjecture}{Conjectures}
\crefname{example}{example}{examples}
\Crefname{example}{Example}{Examples}
\crefname{section}{section}{sections}
\Crefname{section}{Section}{Sections}
\crefname{equation}{equation}{equations}
\Crefname{equation}{Equation}{Equations}

\title{\bfseries
Stein Kernels and Normal Approximation\\
for Log-Concave Bilinear Forms}
\author{%
  \authorwithcontact
    {\textbf{Tianle Liu}}
    {Department of Statistics and Operations Research, UNC Chapel Hill}
    {tianletl@unc.edu}
}
\date{\today}

\begin{document}
\maketitle

\begin{abstract}
Jiang, Lee, and Vempala conjectured that if \(X,Y\in\mathbb R^n\) are
independent isotropic log-concave random vectors, then
\(W_2(\Law(\langle X,Y\rangle),N(0,n))\) is bounded by a universal
constant.  Subject to Theorems~1.2 and~2.5 of Letwin's July 2026 preprint,
we prove this conjecture and a rectangular bilinear-form extension.  For
independent isotropic log-concave \(X\in\mathbb R^m\),
\(Y\in\mathbb R^n\), and nonzero \(B\in\mathbb R^{m\times n}\), put
\[
  r_4(B)
  =\frac{\Tr(B^{\mathsf T}B)^2}
         {\Tr((B^{\mathsf T}B)^2)}.
\]
We construct a nonnegative scalar Stein kernel for
\(X^{\mathsf T}BY/\|B\|_{\HS}\) whose squared \(L^2\) discrepancy is at
most \(20/r_4(B)\), and consequently obtain the same bound for squared
\(2\)-Wasserstein distance to \(N(0,1)\).  The proof develops an exact
covariance identity and deficit decomposition for trace observables of
moment-map Stein kernels, together with a stability theorem for positive
Stein kernels under log-concave approximation.  Taking \(B=I_n\) yields
\[
  W_2^2\!\left(
    \Law\!\left(\frac{\langle X,Y\rangle}{\sqrt n}\right),N(0,1)
  \right)\leq\frac{20}{n},
\]
which is the Jiang--Lee--Vempala conjecture.

\end{abstract}



\section{Introduction}

Jiang, Lee, and Vempala conjectured that there is a universal constant \(C\)
such that
\begin{equation}
  W_2\!\left(\Law(\langle X,Y\rangle),N(0,n)\right)\leq C
  \tag{1.1}
\end{equation}
for every pair of independent isotropic log-concave vectors
\(X,Y\in\R^n\) \cite{JLV}.  Subject to the quadratic-form and moment-map
matrix-energy estimates in Theorems~1.2 and~2.5 of Letwin's July 2026
preprint \cite{Letwin}, we prove this conjecture.  We in fact obtain the
following rectangular extension: if \(X\in\R^m\) and \(Y\in\R^n\) are
independent isotropic log-concave vectors and
\(0\neq B\in\R^{m\times n}\), then, with
\[
  r_4(B)
  :=\frac{\|B\|_{\HS}^4}{\|B\|_{S_4}^4}
  =\frac{\Tr(B^{\mathsf T}B)^2}
         {\Tr((B^{\mathsf T}B)^2)},
\]
we have
\begin{equation}
  W_2^2\!\left(
    \Law\!\left(\frac{X^{\mathsf T}BY}{\|B\|_{\HS}}\right),N(0,1)
  \right)
  \leq\frac{20}{r_4(B)}.
  \tag{1.2}\label{eq:intro-bilinear}
\end{equation}
The parameter \(r_4(B)\) is the participation ratio of the squared singular
values of \(B\), also called the R\'enyi-2 spectral effective rank; it lies
between \(1\) and \(\operatorname{rank}(B)\).  Taking \(m=n\) and \(B=I_n\) in
\eqref{eq:intro-bilinear} gives
\(W_2^2(\Law(\langle X,Y\rangle/\sqrt n),N(0,1))\leq20/n\), equivalently
\(W_2^2(\Law(\langle X,Y\rangle),N(0,n))\leq20\), and thus proves the
Jiang--Lee--Vempala conjecture under the cited estimates.

The proof has two structural steps.  Let \(\tau\) be a moment-map Stein
kernel and let \(A\) be deterministic and symmetric.  In the regular
moment-map model, three source-coordinate integrations by parts give
\begin{equation}
  \Var\!\left(\Tr(A\tau)\right)
  =\Var(Y^{\mathsf T}AY)
   -3\E(AY)^{\mathsf T}\tau(AY)
   +\E\Tr(A\tau A\tau).
  \tag{1.3}\label{eq:intro-trace-identity}
\end{equation}
Fathi's weighted Poincar\'e inequality for the same kernel bounds
\(\E(AY)^{\mathsf T}\tau(AY)\) from below, and hence bounds the negative
middle term in~\eqref{eq:intro-trace-identity} from above.  Polarization
reveals an operator identity with an explicit nonnegative
weighted-Poincar\'e deficit.  Letwin's two estimates then give the covariance
operator bound
\begin{equation}
  \Var\!\left(\Tr(A\tau)\right)\leq4\Tr(A^2).
  \tag{1.4}\label{eq:intro-trace-bound}
\end{equation}

For the second step, put \(S=X^{\mathsf T}BY\),
\(A=B^{\mathsf T}B\), and
\(Q=X^{\mathsf T}B\tau(Y)B^{\mathsf T}X\).  The vector Stein identity for
\(Y\) shows that \(\E[Q\mid S]\) is a nonnegative scalar Stein kernel for
\(S\).  The quadratic-form estimate controls the fluctuation of \(Q\)
conditional on \(Y\); \eqref{eq:intro-trace-bound}, in direction \(A\),
controls its conditional mean.  These contributions are at most
\(16\Tr(A^2)\) and \(4\Tr(A^2)\).  Projection onto \(S\), normalization by
\(\Var S=\Tr A\), and the one-dimensional transport bound by Stein
discrepancy yield \eqref{eq:intro-bilinear}.

\subsection{Related literature}

The geometric setting originates in the isoperimetric conjecture of Kannan,
Lov\'asz, and Simonovits \cite{KLS}.  It predicts a dimension-free comparison
between the Cheeger constant of an isotropic log-concave measure and that of
the worst linear half-space cut.  Besides its intrinsic role in asymptotic
convex geometry, KLS controls concentration, spectral gaps, and the efficiency
of algorithms for sampling log-concave measures; see
\cite{LovaszVempala,KlartagLehecSurvey} for background and applications.

The classical central-limit problem for convex bodies asks whether most
one-dimensional marginals of a high-dimensional isotropic convex body are
approximately Gaussian.  Anttila, Ball, and Perissinaki established an early
quantitative form \cite{ABP}, and Klartag proved a general central limit
theorem and subsequent power-law estimates \cite{KlartagCLT,KlartagPower}.
Those results concern typical linear marginals of one high-dimensional law.
The Jiang--Lee--Vempala conjecture instead randomizes the direction using an
independent log-concave vector and asks for a dimension-free \(W_2\) bound for
their inner product \cite{JLV}.  The same work relates power-law forms of this
conjecture to power-law estimates for the KLS constant, but its reverse
implication is not an endpoint, dimension-free equivalence.

Bilinear and quadratic central limit theorems have a separate classical
history.  De Jong proved qualitative central limit theorems for generalized
quadratic and multilinear forms built from independent scalar coordinates
\cite{deJongQuadratic,deJongMultilinear}; quantitative and multidimensional
versions were later developed by D\"obler and Peccati
\cite{DoblerPeccati}.  Invariance principles for homogeneous sums and
Gaussian-chaos fourth-moment phenomena appear in
\cite{NourdinPeccatiReinert,NourdinPeccatiOptimal}.  For rank one and for
equal-weight sums of independent Gaussian products, the natural law is
variance-gamma rather than Gaussian; Stein's method for that target was
developed by Gaunt \cite{GauntVarianceGamma}.  These theories
typically exploit coordinatewise independence, Hoeffding degeneracy, or an
underlying Gaussian chaos.  Here the coordinates within each factor may be
strongly dependent: the moment-map kernel and the cited estimates
accommodate this within-factor dependence, and the approximation rate is
governed by the spectral participation ratio \(r_4(B)\).

The strongest approaches to KLS have proceeded through stochastic
localization.  On the inverse-Cheeger normalization of the KLS constant,
Eldan's reduction, combined with the then-best thin-shell estimate, gave
\(O(n^{1/3}\sqrt{\log n})\) \cite{Eldan,GuedonMilman}; Lee and Vempala
improved this to \(O(n^{1/4})\) \cite{LeeVempala}; Chen's bound \cite{Chen}
was \(\exp\!\bigl(O(\sqrt{\log n\,\log\log n})\bigr)=n^{o(1)}\);
Klartag--Lehec~obtained a polylogarithmic bound
\cite{KlartagLehecKLS}; and Klartag sharpened it to the published
\(O(\sqrt{\log n})\) estimate \cite{KlartagLog}.  Letwin's recent preprint
improves this to \(O(\log^{1/4}n)\) and proves the two estimates used below
\cite{Letwin}.  None of these KLS bounds is assumed as a black box in our
argument.

Thin-shell estimates form a closely related line of work.  Paouris proved
sharp large-deviation estimates for isotropic log-concave measures
\cite{Paouris}, while Gu\'edon and Milman interpolated those estimates with
then-best thin-shell bounds \cite{GuedonMilman}.  Klartag and Lehec recently
proved a dimension-free thin-shell bound by parallel coupling
\cite{KlartagLehecThinShell}.  Chen and Klartag subsequently gave another
account of the sharp result and proved a sharp Hilbert--Schmidt bound for the
full third-moment tensor \cite{ChenKlartag2026}.  That tensor theorem is not
used here; the matrix-energy estimate needed for general symmetric \(A\) is
the one cited from Letwin.

Our analytic framework is the moment-map construction of Cordero-Erausquin
and Klartag \cite{CEK,KlartagMM}.  Fathi observed that the Hessian of the
moment potential supplies a positive-semidefinite Stein kernel and a weighted
Poincar\'e inequality \cite{Fathi}.  Stein kernels under spectral-gap
assumptions and their discrepancy consequences were further developed in
\cite{CourtadeFathiPananjady}.  Finally, the passage from a scalar
\(L^2\)-Stein discrepancy to \(2\)-Wasserstein distance is taken from
Ledoux, Nourdin, and Peccati \cite{LNP}.  Fathi and Mikulincer proved a
different stability theory for invariant measures, moment measures, and
Stein kernels \cite{FathiMikulincer}.  The compactness theorem in
\cref{sec:approx} is tailored instead to the positive kernels and uniform
matrix energies used here.  The contribution of this paper is to combine
these tools through the exact identity \eqref{eq:intro-trace-identity}, its
operator deficit decomposition, stability under log-concave approximation,
and scalarization for rectangular bilinear forms.

\paragraph{Dependency and scope.}
The recent quantitative results used here are Theorems~1.2 and~2.5 of
\cite{Letwin}, restated in \cref{thm:letwin}; the regular model and its
approximation also use \cite[Lemma~2.2 and Lemma~A.1]{Letwin}.  As checked
on 30 August 2026, these statements occur in the first arXiv version of that
preprint and have not undergone peer review.  Our conclusions are therefore
conditional on those statements.
The argument neither assumes nor proves the KLS conjecture: the scalar
observables \(X^{\mathsf T}BY\) are much more special than arbitrary test
functions of a log-concave vector.

\paragraph{Organization.}
\Cref{sec:inputs} records definitions and the exact estimates used.
\Cref{sec:trace} proves the trace identity, its operator deficit
decomposition, and covariance consequences in the regular model.
\Cref{sec:approx} proves an abstract stability theorem and removes the
regularity assumptions.  \Cref{sec:gclt} proves the bilinear-form theorem
and the Jiang--Lee--Vempala conjecture.  \Cref{sec:checks} examines the role
of spectral participation in Gaussian, rank-one, and product-exponential models;
\cref{sec:limitations} concludes with the scope and limitations.

\section{Preliminaries and quantitative estimates}
\label{sec:inputs}

\subsection{Log-concavity, isotropy, and Stein kernels}

A probability measure on \(\R^n\) is log-concave if it has a density of the
form \(e^{-V}\) on a convex support, where \(V\) is convex in the extended
sense.  A random vector \(Y\) is isotropic if
\[
  \E Y=0,
  \qquad
  \E(YY^{\mathsf T})=I_n.
\]

A measurable matrix field \(\tau:\R^n\to\R^{n\times n}\) is a Stein kernel
for a centered law \(\mu\) if
\begin{equation}
  \E\langle Y,\Phi(Y)\rangle
  =\E\langle\tau(Y),\nabla\Phi(Y)\rangle_{\HS}
  \tag{2.1}
\end{equation}
for every smooth compactly supported vector field \(\Phi\).  Here
\(\langle B,C\rangle_{\HS}=\Tr(B^{\mathsf T}C)\).  If \(Y\) is isotropic,
then testing coordinatewise gives \(\E\tau(Y)=I_n\), whenever the relevant
entries are integrable.

For a centered real random variable \(Z\), a scalar Stein kernel is a
measurable function \(\kappa_Z\) satisfying
\begin{equation}
  \E[Z\phi(Z)]=\E[\kappa_Z(Z)\phi'(Z)]
  \tag{2.2}
\end{equation}
for smooth tests \(\phi\).  If \(\Var Z=1\), the squared \(L^2\) Stein
discrepancy associated with this kernel is
\(\E(\kappa_Z(Z)-1)^2\).

\subsection{The regular moment-map model}

We first work with a regular isotropic log-concave target
\begin{equation}
  \dd\mu(y)=e^{-V(y)}\mathbf 1_K(y)\dd y,
  \tag{2.3}
\end{equation}
where \(K\subset\R^n\) is bounded, open, and convex, and \(V\) is smooth and
convex on a neighborhood of \(\overline K\).  The moment-measure theorem
\cite{CEK,KlartagMM} supplies a smooth strictly convex potential \(\psi\),
unique up to translation, for which
\begin{equation}
  \dd\nu(u)=e^{-\psi(u)}\dd u,
  \qquad
  Y=\nabla\psi(U)\sim\mu,
  \qquad U\sim\nu.
  \tag{2.4}
\end{equation}
The normalization in~(2.4) includes
\(\int e^{-\psi}=1\).  Write
\[
  H=\nabla^2\psi(U).
\]
The map \(\nabla\psi:\R^n\to K\) is a diffeomorphism in the regular model,
and the transported Hessian
\begin{equation}
  \tau(y)=\nabla^2\psi\!\left((\nabla\psi)^{-1}(y)\right)
  \tag{2.5}
\end{equation}
is a symmetric positive-semidefinite Stein kernel \cite{Fathi}.  In
particular, \(H=\tau(Y)\) and \(\E H=I_n\).

The same construction yields the weighted Poincar\'e inequality
\begin{equation}
  \Var_\mu g
  \leq
  \E\langle\tau(Y)\nabla g(Y),\nabla g(Y)\rangle,
  \tag{2.6}\label{eq:weighted-poincare}
\end{equation}
initially for smooth \(g\) and then by closure; see \cite{Fathi}.  It is also
the Brascamp--Lieb inequality for \(e^{-\psi}\), transported by
\(\nabla\psi\): for \(h=g\circ\nabla\psi\), one has
\(\nabla h=H\nabla g(Y)\), which gives exactly~(2.6).

\subsection{Quadratic and moment-map estimates}

The following theorem records the complete quantitative content imported
from \cite{Letwin}.

\begin{theorem}[Letwin's quadratic and moment-map estimates]
\label{thm:letwin}
The cited \cite[Theorems~1.2 and~2.5]{Letwin} state the following estimates.
For
every isotropic log-concave \(Z\in\R^n\) and every deterministic symmetric
matrix \(A\),
\begin{equation}
  \Var(Z^{\mathsf T}AZ)\leq8\Tr(A^2).
  \tag{2.7}\label{eq:letwin-qf}
\end{equation}
In the regular moment-map model above, for every deterministic symmetric
matrix \(A\),
\begin{equation}
  \E\Tr(AHAH)\leq2\Tr(A^2).
  \tag{2.8}\label{eq:letwin-energy}
\end{equation}
\end{theorem}

The constant \(8\) in~\eqref{eq:letwin-qf} is sharp.  The all-\(A\) form
of~\eqref{eq:letwin-energy}, rather than only \(A=I\), is needed for the
trace covariance theorem and the rectangular bilinear-form extension; the
Jiang--Lee--Vempala corollary itself specializes to \(A=I_n\).  This theorem
records the sole quantitative input
from Letwin's preprint; the qualitative regularity and approximation uses
were identified above.  There is no KLS assumption.  In particular,
the weighted inequality~(2.6) is tied to the non-Euclidean moment-map kernel
and does not assert a dimension-free Euclidean Poincar\'e inequality.

\subsection{Transport by scalar Stein discrepancy}

We use the following standard consequence of Stein discrepancy, proved in
the generality needed here by Ledoux, Nourdin, and Peccati \cite{LNP}.

\begin{theorem}[Stein discrepancy controls quadratic transport]
\label{thm:stein-w2}
Let \(Z\) be centered with variance \(1\), and suppose it admits a scalar
Stein kernel \(\kappa_Z\in L^2\).  If \(G\sim N(0,1)\), then
\begin{equation}
  W_2^2(\Law(Z),\Law(G))
  \leq\E(\kappa_Z(Z)-1)^2.
  \tag{2.9}
\end{equation}
\end{theorem}

No density assumption on \(Z\) is needed: the general case follows by
Gaussian regularization in \cite{LNP}.

\section{Covariance transfer for moment-map Stein kernels}
\label{sec:trace}

Fix a deterministic symmetric matrix \(A\).  In the regular model define
\begin{equation}
  f_A=\Tr(AH),
  \qquad
  q_A=Y^{\mathsf T}AY,
  \tag{3.1}
\end{equation}
and the deterministic quantities
\begin{equation}
  T_A=\E(AY)^{\mathsf T}H(AY),
  \qquad
  U_A=\E\Tr(AHAH).
  \tag{3.2}
\end{equation}
Although \(A\) need not be positive-semidefinite, both \(T_A\) and \(U_A\)
are nonnegative: this is immediate after inserting \(H^{1/2}\).

\subsection{Three integrations by parts}

The source measure \(\nu=e^{-\psi}\dd u\) satisfies
\begin{equation}
  \E_\nu\operatorname{div}F(U)
  =\E_\nu\langle F(U),\nabla\psi(U)\rangle
  =\E_\nu\langle F(U),Y\rangle
  \tag{3.3}
\end{equation}
for compactly supported smooth vector fields \(F\).

\begin{lemma}[Exact trace identity]
\label{lem:identity}
In the regular moment-map model, suppose that
\(\E\Tr(H^2)<\infty\).  Then
\begin{equation}
  \boxed{
  \Var(f_A)=\Var(q_A)-3T_A+U_A.}
  \tag{3.4}
\end{equation}
\end{lemma}

\begin{proof}
We give the source-coordinate calculation before discussing cutoffs.  Since
\(Y=\nabla\psi\) and \(H=\nabla^2\psi\),
\begin{equation}
  \operatorname{div}(AY)=\Tr(AH)=f_A,
  \qquad
  \nabla q_A=2HAY.
  \tag{3.5}
\end{equation}

Apply~(3.3) first to the vector field \(f_AAY\).  This gives
\begin{equation}
  \E f_A^2
  =\E(f_Aq_A)-\E\langle AY,\nabla f_A\rangle.
  \tag{3.6}
\end{equation}
Applying~(3.3) to \(q_AAY\) and using~(3.5) gives
\begin{equation}
  \E(f_Aq_A)=\E q_A^2-2T_A.
  \tag{3.7}
\end{equation}

For the third identity, apply~(3.3) to \(AHA Y\).  Symmetry of the third
derivatives of \(\psi\) implies
\begin{equation}
  \operatorname{div}(AHA Y)
  =\Tr(AHAH)+\langle AY,\nabla f_A\rangle.
  \tag{3.8}
\end{equation}
Indeed, in index notation,
\[
  \partial_i\!\left(A_{ij}H_{jk}A_{k\ell}Y_\ell\right)
  =(AY)_k\,\partial_k f_A+\Tr(AHAH),
\]
where symmetry of the third derivatives is used in the first term.
On the other hand,
symmetry of \(A\) gives
\(\langle AHA Y,Y\rangle=(AY)^{\mathsf T}H(AY)\).  Therefore
\begin{equation}
  \E\langle AY,\nabla f_A\rangle=T_A-U_A.
  \tag{3.9}
\end{equation}
Substitution of~(3.7) and~(3.9) into~(3.6) yields
\begin{equation}
  \E f_A^2=\E q_A^2-3T_A+U_A.
  \tag{3.10}
\end{equation}
The displayed calculation is formal until the cutoff passage below, which
also supplies the centering needed to pass from~(3.10) to~(3.4).

Choose a radial \(\chi\in C_c^\infty(\R^n)\), with \(0\leq\chi\leq1\),
equal to one on the unit ball and nonincreasing along rays, and put
\(\chi_R(u)=\chi(u/R)\), so that
\(\chi_R\uparrow1\) and \(\lvert\nabla\chi_R\rvert\leq C/R\).  Because the
regular target \(K\) is bounded, there is an \(R_K<\infty\) such that
\(\lvert Y\rvert\leq R_K\).  Applying~(3.3) first to \(\chi_RY\) gives
\[
 \E[\chi_R\Tr H]
 =\E[\chi_R|Y|^2]-\E\langle Y,\nabla\chi_R\rangle.
\]
Since \(H\succeq0\), monotone convergence and the boundedness of \(Y\)
show that \(\E\Tr H=n\).  In particular,
\(\E\|H\|_{\mathrm{op}}\leq n\).  Applying the same argument to
\(\chi_RAY\) now gives
\begin{equation}
  \E f_A=\E q_A=\Tr A,
  \tag{3.10a}
\end{equation}
because \(\lvert f_A\rvert\leq\|A\|_{\mathrm{op}}\Tr H\) and the new
boundary term is \(O_A(R^{-1})\).

We next replace the three noncompact vector
fields above by
\[
  \chi_R f_AAY,
  \qquad
  \chi_R q_AAY,
  \qquad
  \chi_R AHA Y,
\]
and put
\[
\begin{aligned}
 t_A&=(AY)^{\mathsf T}H(AY),
 &u_A&=\Tr(AHAH),\\
 B_{1,R}&=\E\bigl[f_A\langle AY,\nabla\chi_R\rangle\bigr],\\
 B_{2,R}&=\E\bigl[q_A\langle AY,\nabla\chi_R\rangle\bigr],\\
 B_{3,R}&=\E\bigl[\langle AHA Y,\nabla\chi_R\rangle\bigr].
\end{aligned}
\]
Applying~(3.3) at finite \(R\), rather than to the uncut fields, gives
\[
\begin{aligned}
 \E[\chi_R f_A^2]
   &=\E[\chi_R f_Aq_A]
     -\E[\chi_R\langle AY,\nabla f_A\rangle]-B_{1,R},\\
 \E[\chi_R f_Aq_A]
   &=\E[\chi_R q_A^2]-2\E[\chi_R t_A]-B_{2,R},\\
 \E[\chi_R\langle AY,\nabla f_A\rangle]
   &=\E[\chi_R t_A]-\E[\chi_R u_A]-B_{3,R}.
\end{aligned}
\]
Combining these three identities first yields
\[
 \E[\chi_R f_A^2]
 =\E[\chi_R q_A^2]-3\E[\chi_R t_A]
   +\E[\chi_R u_A]-B_{1,R}-B_{2,R}+B_{3,R}.
\]
The boundary terms vanish without any boundedness assumption on \(H\).
Indeed,
\[
 \lvert f_A\rvert\leq\|A\|_{\mathrm{op}}\Tr H,
 \qquad
 \lvert AHA Y\rvert
 \leq\|A\|_{\mathrm{op}}^2R_K\Tr H,
\]
while \(q_A\) and \(AY\) are bounded.  Hence each
\(B_{j,R}=O_A(R^{-1})\) using only \(\E\Tr H=n\).
For the bulk terms, positivity of \(H\) and
\(\E\Tr(H^2)<\infty\) give integrable dominators: in particular,
\[
 f_A^2\leq n\|A\|_{\mathrm{op}}^2\Tr(H^2),
 \qquad
 u_A\leq\|A\|_{\mathrm{op}}^2\Tr(H^2),
 \qquad
 t_A\leq\|A\|_{\mathrm{op}}^2R_K^2\Tr H.
\]
Dominated convergence therefore proves~(3.10), and~(3.10a) gives the
centered identity~(3.4).  All appearances of \(\nabla H\) have cancelled in
the finite-\(R\) combination before \(R\to\infty\).  Thus the argument uses
neither global boundedness of the Hessian nor an unproved third-derivative
integrability estimate.  Under the assumptions of
\cref{thm:trace}, the required second moment follows immediately
from~\eqref{eq:letwin-energy} with \(A=I_n\).
\end{proof}

\subsection{Covariance-operator form and trace bound}

Let \(\mathsf S_n\) be the Hilbert space of real symmetric \(n\times n\)
matrices with Hilbert--Schmidt inner product
\[
  \langle A,B\rangle_{\HS}=\Tr(AB).
\]
For \(A,B\in\mathsf S_n\), define the symmetric bilinear forms
\[
\begin{aligned}
  \Gamma_{\rm tr}(A,B)
    &=\Cov(f_A,f_B),\\
  \Gamma_{\rm quad}(A,B)
    &=\Cov(q_A,q_B),\\
  \mathsf T(A,B)
    &=\E(AY)^{\mathsf T}H(BY),\\
  \mathsf U(A,B)
    &=\E\Tr(AHBH).
\end{aligned}
\]
Thus \(\mathsf T(A,A)=T_A\) and \(\mathsf U(A,A)=U_A\).
The symmetry of \(\mathsf T\) follows from the symmetry of \(H\), while the
symmetry of \(\mathsf U\) follows by transposition and cyclicity of the
trace.  We use the same notation for these forms and their Riesz
representatives on \(\mathsf S_n\).  In particular,
\(I_{\mathsf S_n}\) denotes the identity operator, whose associated
bilinear form is
\[
  I_{\mathsf S_n}(A,B)=\Tr(AB).
\]
For symmetric bilinear forms, \(\mathsf A\preceq\mathsf B\) means that
\((\mathsf B-\mathsf A)(A,A)\geq0\) for every \(A\in\mathsf S_n\).

\begin{proposition}[Polarized identity and weighted-Poincar\'e deficit]
\label{prop:trace-operator}
In the regular moment-map model,
\begin{equation}
  \boxed{
  \Gamma_{\rm tr}(A,B)
  =
  \Gamma_{\rm quad}(A,B)
  -3\mathsf T(A,B)
  +\mathsf U(A,B)}
  \tag{3.11}\label{eq:polarized-trace}
\end{equation}
for every \(A,B\in\mathsf S_n\).  Moreover, the bilinear form
\begin{equation}
  \mathsf D
  :=
  4\mathsf T-\Gamma_{\rm quad}
  \tag{3.12}\label{eq:wp-deficit}
\end{equation}
is positive-semidefinite, and
\begin{equation}
  \boxed{
  \Gamma_{\rm tr}
  =
  \frac14\Gamma_{\rm quad}
  +\mathsf U
  -\frac34\mathsf D.}
  \tag{3.13}\label{eq:operator-deficit}
\end{equation}
\end{proposition}

\begin{proof}
Apply \cref{lem:identity} with \(A+B\), subtract its instances with \(A\)
and \(B\), and divide by two.  This gives
\eqref{eq:polarized-trace}.  No further integration by parts or cutoff
passage is required.

The weighted Poincar\'e inequality~\eqref{eq:weighted-poincare}, applied to
\(q_A(y)=y^{\mathsf T}Ay\), gives
\[
  \Gamma_{\rm quad}(A,A)
  =\Var(q_A)
  \leq4\E(AY)^{\mathsf T}H(AY)
  =4\mathsf T(A,A).
\]
Consequently \(\mathsf D\succeq0\).  Solving
\eqref{eq:wp-deficit} for \(\mathsf T\) and substituting into
\eqref{eq:polarized-trace} proves \eqref{eq:operator-deficit}.
\end{proof}

The representation \eqref{eq:operator-deficit} separates the two
quantitative inputs from the nonnegative slack in the weighted Poincar\'e
inequality.  The following statement records this transfer with arbitrary
constants.

\begin{theorem}[Trace bound with modular constants]
\label{thm:trace}
Suppose that, for some \(C_q,C_{\rm en}\geq0\), one has
\[
  \Var(Y^{\mathsf T}AY)
  \leq C_q\Tr(A^2),
  \qquad
  \E\Tr(AHAH)
  \leq C_{\rm en}\Tr(A^2)
\]
for every \(A\in\mathsf S_n\).  Then
\begin{equation}
  0\preceq\Gamma_{\rm tr}
  \preceq
  \left(\frac{C_q}{4}+C_{\rm en}\right)I_{\mathsf S_n}.
  \tag{3.14}\label{eq:modular-trace-bound}
\end{equation}
More precisely, the full slack has the positive decomposition
\begin{equation}
\begin{aligned}
  \left(\frac{C_q}{4}+C_{\rm en}\right)I_{\mathsf S_n}
  -\Gamma_{\rm tr}
  ={}&
  \frac14\left(C_qI_{\mathsf S_n}-\Gamma_{\rm quad}\right)\\
  &+\left(C_{\rm en}I_{\mathsf S_n}-\mathsf U\right)
  +\frac34\mathsf D.
\end{aligned}
  \tag{3.15}\label{eq:modular-slack}
\end{equation}

Under the estimates in \cref{thm:letwin}, \(C_q=8\) and
\(C_{\rm en}=2\).  Hence
\begin{equation}
\begin{aligned}
  4I_{\mathsf S_n}-\Gamma_{\rm tr}
  ={}&
  \frac14\left(8I_{\mathsf S_n}-\Gamma_{\rm quad}\right)
  +\left(2I_{\mathsf S_n}-\mathsf U\right)
  +\frac34\mathsf D
  \succeq0,
\end{aligned}
  \tag{3.16}\label{eq:letwin-trace-slack}
\end{equation}
and, for every deterministic symmetric \(A\),
\begin{equation}
  \boxed{
  \Var\!\left(\Tr(AH)\right)
  \leq4\Tr(A^2).}
  \tag{3.17}\label{eq:trace-bound}
\end{equation}
In particular,
\begin{equation}
  \Var(\Tr H)\leq4n.
  \tag{3.18}\label{eq:scalar-trace-bound}
\end{equation}
\end{theorem}

\begin{proof}
The two assumed estimates are precisely
\[
  \Gamma_{\rm quad}\preceq C_qI_{\mathsf S_n},
  \qquad
  \mathsf U\preceq C_{\rm en}I_{\mathsf S_n}.
\]
Since \(\mathsf D\succeq0\), \eqref{eq:operator-deficit} gives the upper
bound in \eqref{eq:modular-trace-bound}; the lower bound holds because
\(\Gamma_{\rm tr}\) is a covariance operator.  Rearranging
\eqref{eq:operator-deficit} gives the exact decomposition
\eqref{eq:modular-slack}.  Finally, substitute the constants from
\eqref{eq:letwin-qf} and \eqref{eq:letwin-energy}.  Evaluating the resulting
operator inequality in the direction \(A\) proves
\eqref{eq:trace-bound}, and \(A=I_n\) gives
\eqref{eq:scalar-trace-bound}.
\end{proof}

Equation~\eqref{eq:modular-slack} also records the algebraic equality
conditions.  Equality in \eqref{eq:modular-trace-bound} in a fixed
direction \(A\) requires equality in the quadratic-form estimate, the
matrix-energy estimate, and the weighted Poincar\'e inequality in that
direction.  We do not infer from this a geometric classification of
extremizing measures.

\begin{corollary}[Finite-family covariance domination]
\label{cor:covariance}
Assume the estimates in \cref{thm:letwin}.  Let
\(A_1,\ldots,A_k\in\mathsf S_n\), and define
\[
  K_{ij}
  =
  \Cov\!\left(\Tr(A_iH),\Tr(A_jH)\right),
  \qquad
  G_{ij}
  =
  \Tr(A_iA_j).
\]
Then
\begin{equation}
  0\preceq K\preceq4G.
  \tag{3.19}\label{eq:finite-covariance}
\end{equation}
In particular, if \(A_1,\ldots,A_k\) are Hilbert--Schmidt orthonormal,
then \(K\preceq4I_k\).  For two arbitrary directions this also gives
\[
  \left|
  \Cov\!\left(\Tr(AH),\Tr(BH)\right)
  \right|
  \leq
  4\sqrt{\Tr(A^2)\Tr(B^2)}.
\]
\end{corollary}

\begin{proof}
For \(c=(c_1,\ldots,c_k)\in\R^k\), apply
\eqref{eq:trace-bound} to \(A=\sum_{i=1}^kc_iA_i\).  This gives
\[
  c^{\mathsf T}Kc
  \leq
  4\Tr\!\left(\left(\sum_{i=1}^kc_iA_i\right)^2\right)
  =
  4c^{\mathsf T}Gc.
\]
The lower bound follows because \(K\) is a covariance matrix.  The final
display follows from Cauchy--Schwarz for the centered trace observables and
\eqref{eq:trace-bound}.
\end{proof}

\begin{corollary}[Global spectral budget]
\label{cor:spectral-budget}
Under the same assumptions, the covariance operator satisfies
\begin{equation}
  \Tr_{\mathsf S_n}\Gamma_{\rm tr}
  =
  \E\|H-I_n\|_{\HS}^2
  =
  \E\Tr(H^2)-n
  \leq n.
  \tag{3.20}\label{eq:spectral-budget}
\end{equation}
Consequently, every Hilbert--Schmidt orthonormal family
\(A_1,\ldots,A_k\in\mathsf S_n\) satisfies
\begin{equation}
  \sum_{i=1}^k
  \Var\!\left(\Tr(A_iH)\right)
  \leq\min\{4k,n\},
  \qquad
  \|\Gamma_{\rm tr}\|_{\HS(\mathsf S_n)}^2\leq4n.
  \tag{3.21}\label{eq:spectral-consequences}
\end{equation}
\end{corollary}

\begin{proof}
Let \((E_\alpha)_\alpha\) be a Hilbert--Schmidt orthonormal basis of
\(\mathsf S_n\).  Since \(\E H=I_n\),
\[
\begin{aligned}
  \Tr_{\mathsf S_n}\Gamma_{\rm tr}
  &=
  \sum_\alpha
  \E\langle E_\alpha,H-I_n\rangle_{\HS}^2\\
  &=
  \E\|H-I_n\|_{\HS}^2
  =
  \E\Tr(H^2)-n.
\end{aligned}
\]
The case \(A=I_n\) of \eqref{eq:letwin-energy} bounds the last expression
by \(n\).  The first inequality in \eqref{eq:spectral-consequences} follows
by combining this trace budget with
\(\|\Gamma_{\rm tr}\|_{\rm op}\leq4\).  If
\((\lambda_\alpha)_\alpha\) are the eigenvalues of
\(\Gamma_{\rm tr}\), then \(0\leq\lambda_\alpha\leq4\), and therefore
\[
  \|\Gamma_{\rm tr}\|_{\HS(\mathsf S_n)}^2
  =
  \sum_\alpha\lambda_\alpha^2
  \leq4\sum_\alpha\lambda_\alpha
  \leq4n.
\]
\end{proof}

\begin{remark}[Deficit bookkeeping for product exponentials]
Let \(Y_i=E_i-1\), where the \(E_i\) are independent
\(\operatorname{Exp}(1)\) variables.  The positive Stein kernel is
\(H=\operatorname{diag}(E_1,\ldots,E_n)\).  Direct calculation gives
\[
  \Gamma_{\rm tr}(A,B)=\sum_{i=1}^n A_{ii}B_{ii}.
\]
Thus \(\Gamma_{\rm tr}\) is the orthogonal projection of
\(\mathsf S_n\) onto its diagonal subspace, and
\(\Tr_{\mathsf S_n}\Gamma_{\rm tr}=n\).  Hence the global budget
\eqref{eq:spectral-budget} is sharp, although the directional constant
in \eqref{eq:trace-bound} is not.

In the direction \(A=I_n\),
\[
\begin{aligned}
  \Gamma_{\rm quad}(I_n,I_n)&=8n,
  &\mathsf T(I_n,I_n)&=3n,\\
  \mathsf U(I_n,I_n)&=2n,
  &\mathsf D(I_n,I_n)&=4n,
  &\Gamma_{\rm tr}(I_n,I_n)&=n.
\end{aligned}
\]
The quadratic and matrix-energy deficits in
\eqref{eq:letwin-trace-slack} therefore vanish in this direction, while
the weighted-Poincar\'e contribution is
\(\frac34\mathsf D(I_n,I_n)=3n\).  It accounts exactly for the gap between
the universal upper bound \(4n\) and the actual trace variance \(n\).
\end{remark}

The exact mixed identity in \cref{prop:trace-operator} is asserted for the
regular moment-map Hessian.  In the next section we pass its covariance and
energy consequences to a limiting positive Stein kernel; we do not claim
that every mixed term in \eqref{eq:polarized-trace} survives for an
arbitrary nonsmooth kernel.

\section{Stability of positive Stein kernels}
\label{sec:approx}

The exact identity of \cref{lem:identity} uses the smooth moment-map
coordinates.  Its consequences, however, are stable under weak
approximation.  We first record a moment estimate that is useful for this
passage.  Write \(\mathsf S_n\) for the space of real symmetric
\(n\times n\) matrices.

\begin{lemma}[Fourth moments from a positive Stein kernel]
\label{lem:stein-fourth}
Let \(Y\) be centered with finite second moments, and suppose that \(Y\)
admits a symmetric positive-semidefinite Stein kernel
\(\tau\in L^2\).  Then
\begin{equation}
  \E|Y|^4
  \leq 9n\,\E\|\tau(Y)\|_{\HS}^2.
  \tag{4.1}\label{eq:stein-fourth}
\end{equation}
\end{lemma}

\begin{proof}
Fix a coordinate \(i\) and \(R>0\), and define the odd \(C^1\) function
\[
  h_R(t)
  =3\int_0^t\min\{s^2,R^2\}\,\dd s.
\]
Thus
\[
  h_R'(t)=3\min\{t^2,R^2\},
  \qquad
  t h_R(t)\geq\min\{t^4,R^4\},
  \qquad
  |h_R(t)|\leq3R^2|t|.
\]
To justify this test, choose \(\chi\in C_c^\infty(\R^n)\) equal to one on
the unit ball and zero outside the ball of radius two, put
\(\chi_L(x)=\chi(x/L)\), and apply the Stein identity to smooth
approximations of
\(\Phi_{L,R}(x)=e_i h_R(x_i)\chi_L(x)\).  Removing the smooth
approximation gives
\begin{align*}
  \E\!\left[Y_i h_R(Y_i)\chi_L(Y)\right]
  &=
  \E\!\left[\tau_{ii}(Y)h_R'(Y_i)\chi_L(Y)\right]\\
  &\quad+
  \E\!\left[
    h_R(Y_i)\sum_{j=1}^n\tau_{ij}(Y)\partial_j\chi_L(Y)
  \right].
\end{align*}
The first two terms are dominated by integrable multiples of \(Y_i^2\)
and \(\tau_{ii}\), respectively.  On the support of \(\nabla\chi_L\),
the absolute value of the last integrand is bounded by
\[
  C R^2\|\tau(Y)\|_{\HS}\mathbf 1_{\{|Y|\geq L\}},
\]
which tends to zero in \(L^1\).  Hence
\[
  \E\!\left[Y_i h_R(Y_i)\right]
  =
  3\E\!\left[\tau_{ii}(Y)\min\{Y_i^2,R^2\}\right].
\]
If \(M_{i,R}=\E\min\{Y_i^4,R^4\}\), positivity of \(\tau\) and
Cauchy--Schwarz give
\[
  M_{i,R}
  \leq3\bigl(\E\tau_{ii}^2\bigr)^{1/2}M_{i,R}^{1/2},
  \qquad
  M_{i,R}\leq9\E\tau_{ii}^2.
\]
Letting \(R\to\infty\) and summing over coordinates yields
\[
  \E|Y|^4
  \leq n\sum_{i=1}^n\E Y_i^4
  \leq9n\sum_{i=1}^n\E\tau_{ii}^2
  \leq9n\,\E\|\tau\|_{\HS}^2.
\]
\end{proof}

\begin{proposition}[Weak stability of positive \(L^2\) Stein kernels]
\label{prop:stein-stability}
Let \(\mu_k\Rightarrow\mu\) weakly on \(\R^n\), where each \(\mu_k\)
is centered and has finite second moments.  Let \(Y_k\sim\mu_k\), and
suppose that \(\mu_k\) admits a symmetric positive-semidefinite Stein
kernel \(\tau_k\) such that
\begin{equation}
  L:=\sup_k\E\|\tau_k(Y_k)\|_{\HS}^2<\infty.
  \tag{4.2}\label{eq:uniform-kernel-energy}
\end{equation}
Then \(\mu\) is centered, \(\mu_k\to\mu\) in \(W_2\), and \(\mu\)
admits a symmetric positive-semidefinite Stein kernel \(\tau\) satisfying
\begin{equation}
  \E\tau(Y)=\Cov(Y),
  \qquad
  \E\|\tau(Y)\|_{\HS}^2
  \leq
  \liminf_{k\to\infty}\E\|\tau_k(Y_k)\|_{\HS}^2,
  \tag{4.3}\label{eq:stable-basic}
\end{equation}
where \(Y\sim\mu\).  The same kernel may be chosen with the following
additional permanence properties.

\begin{enumerate}[label=\textup{(\roman*)}]
\item
If \(F:\mathsf S_n\to[0,\infty]\) is convex and lower semicontinuous and
\(\sup_k\E F(\tau_k(Y_k))\leq C_F\), then
\[
  \E F(\tau(Y))\leq C_F.
\]

\item
If, for every \(A\in\mathsf S_n\),
\[
  \Var\!\left(\Tr(A\tau_k(Y_k))\right)
  \leq C_{\rm tr}\Tr(A^2),
\]
then
\begin{equation}
  \Var\!\left(\Tr(A\tau(Y))\right)
  \leq C_{\rm tr}\Tr(A^2)
  \qquad(A\in\mathsf S_n).
  \tag{4.4}\label{eq:stable-trace}
\end{equation}

\item
If, for every \(A\in\mathsf S_n\),
\[
  \E\Tr(A\tau_k(Y_k)A\tau_k(Y_k))
  \leq C_{\rm en}\Tr(A^2),
\]
then
\begin{equation}
  \E\Tr(A\tau(Y)A\tau(Y))
  \leq C_{\rm en}\Tr(A^2)
  \qquad(A\in\mathsf S_n).
  \tag{4.5}\label{eq:stable-energy}
\end{equation}

\item
If a constant \(C_{\rm wp}\) satisfies
\[
  \Var_{\mu_k}g
  \leq
  C_{\rm wp}\,
  \E\langle
    \tau_k(Y_k)\nabla g(Y_k),\nabla g(Y_k)
  \rangle
\]
for every \(k\) and every \(g\in C_c^\infty(\R^n)\), then
\begin{equation}
  \Var_\mu g
  \leq
  C_{\rm wp}\,
  \E\langle
    \tau(Y)\nabla g(Y),\nabla g(Y)
  \rangle
  \qquad(g\in C_c^\infty(\R^n)).
  \tag{4.6}\label{eq:stable-weighted-poincare}
\end{equation}
\end{enumerate}
\end{proposition}

\begin{proof}
Cutoff approximations to the linear vector fields \(x\mapsto e_i x_j\)
give
\begin{equation}
  \E\tau_k(Y_k)=\E(Y_kY_k^{\mathsf T})=\Cov(Y_k).
  \tag{4.7}\label{eq:kernel-mean-covariance}
\end{equation}
By \cref{lem:stein-fourth},
\[
  \sup_k\E|Y_k|^4\leq9nL.
\]
Thus \((|Y_k|^2)_k\) is uniformly integrable.  Weak convergence implies
convergence of the first and second moments, and therefore
\begin{equation}
  W_2(\mu_k,\mu)\longrightarrow0.
  \tag{4.8}\label{eq:w2-from-kernels}
\end{equation}

Choose a subsequence along which the liminf in \eqref{eq:stable-basic} is
attained.  The joint laws of \((Y_k,\tau_k(Y_k))\) are tight.  Passing to
a further subsequence, let
\[
  (Y_k,\tau_k(Y_k))\Rightarrow(Y,M).
\]
The cone of symmetric positive-semidefinite matrices is closed, so \(M\)
is symmetric and positive-semidefinite, and Portmanteau gives
\[
  \E\|M\|_{\HS}^2
  \leq\liminf_k\E\|\tau_k(Y_k)\|_{\HS}^2.
\]
For \(\Phi\in C_c^\infty(\R^n;\R^n)\), truncate the matrix coordinate
and pass to the joint weak limit in the Stein identity.  The uniform
\(L^2\) bound supplies uniform integrability, and hence
\begin{equation}
  \E\langle Y,\Phi(Y)\rangle
  =\E\langle M,\nabla\Phi(Y)\rangle_{\HS}.
  \tag{4.9}\label{eq:joint-limit-stein}
\end{equation}
A weak limit of graphs need not remain a graph, so define
\begin{equation}
  \tau(Y)=\E[M\mid Y].
  \tag{4.10}\label{eq:barycentric-kernel}
\end{equation}
Conditional expectation preserves symmetry and positive semidefiniteness,
and \eqref{eq:joint-limit-stein} becomes the Stein identity for \(\tau\).
Uniform integrability of the matrix variables, together with
\eqref{eq:kernel-mean-covariance} and \eqref{eq:w2-from-kernels}, gives
\[
  \E M=\lim_k\E\tau_k(Y_k)=\Cov(Y).
\]
Conditional Jensen and Portmanteau now prove \eqref{eq:stable-basic}.
The same two tools prove part~\textup{(i)}.

For part~\textup{(ii)}, conditional variance contraction and Portmanteau
applied to the centered square give the claim; the centering constants
converge by \eqref{eq:kernel-mean-covariance} and the second-moment
convergence already proved.  Explicitly,
\[
  \Var\!\left(\Tr(A\tau(Y))\right)
  \leq\Var(\Tr(AM))
  \leq\liminf_k\Var\!\left(\Tr(A\tau_k(Y_k))\right).
\]
This proves \eqref{eq:stable-trace}.

For the matrix energy, first suppose that \(B\succeq0\).  The map
\[
  T\longmapsto\Tr(BTBT)
  =\|B^{1/2}TB^{1/2}\|_{\HS}^2
\]
is continuous and convex.  Conditional Jensen and Portmanteau therefore
pass the assumed estimate to \(\tau\).  For arbitrary
\(A\in\mathsf S_n\), write \(A=A_+-A_-\) and
\(|A|=A_++A_-\).  If \(T\succeq0\), then
\begin{align}
  \Tr(|A|T|A|T)-\Tr(ATAT)
  &=4\Tr(A_+TA_-T)\notag\\
  &=4\|A_-^{1/2}TA_+^{1/2}\|_{\HS}^2
  \geq0.
  \tag{4.11}\label{eq:absolute-value-energy}
\end{align}
Applying the positive-semidefinite case with \(B=|A|\) gives
\[
  \E\Tr(A\tau A\tau)
  \leq\E\Tr(|A|\tau|A|\tau)
  \leq C_{\rm en}\Tr(|A|^2)
  =C_{\rm en}\Tr(A^2),
\]
which is part~\textup{(iii)}.

Finally, fix \(g\in C_c^\infty(\R^n)\).  Weak convergence gives
\(\Var_{\mu_k}g\to\Var_\mu g\).  The function
\[
  (y,T)\longmapsto
  \langle T\nabla g(y),\nabla g(y)\rangle
\]
is continuous and bounded in absolute value by
\(\|\nabla g\|_\infty^2\|T\|_{\HS}\).  Uniform integrability of the
matrix variables lets us pass to the joint limit; conditioning on \(Y\)
then replaces \(M\) by \(\tau(Y)\).  This proves
\eqref{eq:stable-weighted-poincare}.
\end{proof}

\begin{remark}[Why the absolute-value step is necessary]
For indefinite \(A\), the map \(T\mapsto\Tr(ATAT)\) is not convex even on
the positive-semidefinite cone.  For example, take
\[
  A=\begin{pmatrix}1&0\\0&-1\end{pmatrix},
  \qquad
  T_\pm=\begin{pmatrix}1&\pm c\\\pm c&1\end{pmatrix},
  \qquad 0<c<1.
\]
Then \(T_\pm\succeq0\), their midpoint is \(I_2\), and
\[
  \Tr(AI_2AI_2)=2
  >2(1-c^2)
  =\frac12\Tr(AT_+AT_+)+\frac12\Tr(AT_-AT_-).
\]
Thus direct conditional Jensen would be invalid.  Positivity of the kernel
and \eqref{eq:absolute-value-energy} preserve the full all-symmetric
estimate.
\end{remark}

\begin{corollary}[Log-concave approximation]
\label{prop:limit}
Assume the estimates in \cref{thm:letwin}.  Every isotropic log-concave
law \(\mu\) on \(\R^n\) admits a symmetric positive-semidefinite Stein
kernel \(\tau\) satisfying, simultaneously for every
\(A\in\mathsf S_n\),
\begin{align}
  \E\tau&=I_n,
  \tag{4.12}\label{eq:limit-kernel-mean}\\
  \E\Tr(A\tau A\tau)
  &\leq2\Tr(A^2),
  \tag{4.13}\label{eq:limit-kernel-energy}\\
  \Var\!\left(\Tr(A\tau)\right)
  &\leq4\Tr(A^2).
  \tag{4.14}\label{eq:limit-kernel-trace}
\end{align}
It also satisfies
\begin{equation}
  \Var_\mu g
  \leq
  \E\langle\tau(Y)\nabla g(Y),\nabla g(Y)\rangle
  \qquad(g\in C_c^\infty(\R^n)).
  \tag{4.15}\label{eq:limit-kernel-weighted-poincare}
\end{equation}
In particular,
\begin{equation}
  \E\Tr(\tau^2)\leq2n,
  \qquad
  \Var(\Tr\tau)\leq4n.
  \tag{4.16}\label{eq:limit-scalar-bounds}
\end{equation}
\end{corollary}

\begin{proof}
The smoothing, truncation, recentering, and whitening construction of
\cite[Lemma~A.1]{Letwin} provides regular isotropic log-concave laws
\(\mu_k\Rightarrow\mu\).  Let \(\tau_k\) be their moment-map Stein kernels.
Fathi's construction gives the Stein identity and the weighted Poincar\'e
inequality \eqref{eq:weighted-poincare}.  Letwin's matrix-energy estimate
\eqref{eq:letwin-energy} gives
\[
  \E\Tr(A\tau_kA\tau_k)\leq2\Tr(A^2),
\]
and \cref{thm:trace} gives
\[
  \Var\!\left(\Tr(A\tau_k)\right)\leq4\Tr(A^2).
\]
Taking \(A=I_n\) in the energy estimate supplies the uniform \(L^2\)
hypothesis of \cref{prop:stein-stability}.  The conclusions now follow
from that proposition with
\[
  C_{\rm en}=2,
  \qquad
  C_{\rm tr}=4,
  \qquad
  C_{\rm wp}=1.
\]
\end{proof}

\begin{remark}
The kernel furnished by \cref{prop:limit} is a weak-limit kernel.  It need
not be the pointwise Hessian of a nonsmooth moment potential, it need not be
canonical, and it may depend on the chosen subsequence.  Only its Stein
identity and the displayed inequalities are used below.
\end{remark}

\section{Normal approximation for log-concave bilinear forms}
\label{sec:gclt}

We now prove the rectangular statement announced in the introduction.
For a nonzero matrix \(B\), define its spectral participation ratio by
\begin{equation}
  r_4(B)
  :=\frac{\|B\|_{\HS}^4}{\|B\|_{S_4}^4}
  =\frac{\Tr(B^{\mathsf T}B)^2}
         {\Tr((B^{\mathsf T}B)^2)}.
  \tag{5.1}\label{eq:effective-rank}
\end{equation}

\begin{theorem}[Log-concave bilinear forms]
\label{thm:bilinear}
Assume the estimates in \cref{thm:letwin}.  Let
\(X\in\R^m\) and \(Y\in\R^n\) be independent isotropic log-concave random
vectors, let \(0\neq B\in\R^{m\times n}\), and set
\[
  S=X^{\mathsf T}BY,
  \qquad
  \sigma^2=\|B\|_{\HS}^2.
\]
Then \(Z=S/\sigma\) admits a nonnegative scalar Stein kernel
\(\kappa_Z\) such that
\begin{equation}
  \E(\kappa_Z(Z)-1)^2
  \leq\frac{20}{r_4(B)}.
  \tag{5.2}\label{eq:bilinear-stein}
\end{equation}
Consequently, for \(G\sim N(0,1)\),
\begin{equation}
  \boxed{
  W_2^2\!\left(\Law(Z),\Law(G)\right)
  \leq\frac{20}{r_4(B)}.}
  \tag{5.3}\label{eq:bilinear-w2}
\end{equation}
Equivalently, if \(G_\sigma\sim N(0,\sigma^2)\), then
\begin{equation}
  W_2^2\!\left(\Law(S),\Law(G_\sigma)\right)
  \leq
  20\frac{\Tr((B^{\mathsf T}B)^2)}{\Tr(B^{\mathsf T}B)}.
  \tag{5.4}\label{eq:bilinear-w2-unnormalized}
\end{equation}
\end{theorem}

\begin{proof}
Choose for \(Y\) a positive-semidefinite Stein kernel \(\tau(Y)\) furnished
by \cref{prop:limit}, and put
\begin{equation}
  A=B^{\mathsf T}B,
  \qquad
  Q=X^{\mathsf T}B\tau(Y)B^{\mathsf T}X.
  \tag{5.5}\label{eq:bilinear-prekernel}
\end{equation}
Independence and isotropy give
\begin{equation}
  \E S=0,
  \qquad
  \Var S=\Tr A=\sigma^2,
  \qquad
  \E Q=\Tr A.
  \tag{5.6}\label{eq:bilinear-means}
\end{equation}

Condition on \(X\) and apply the vector Stein identity for \(Y\) to
\[
  \Phi_x(y)=B^{\mathsf T}x\,\phi(x^{\mathsf T}By).
\]
Starting with compactly supported smooth \(\phi\), radial cutoffs in \(y\),
and bounded \(x\), one may remove both restrictions using \(\tau\in L^2\),
truncation of \(X\), and the \(L^2\) estimate for \(Q\) below.  Since
\[
  \nabla_y\Phi_x(y)
  =B^{\mathsf T}xx^{\mathsf T}B\,
    \phi'(x^{\mathsf T}By),
\]
the result is
\begin{equation}
  \E[S\phi(S)]=\E[Q\phi'(S)].
  \tag{5.7}\label{eq:bilinear-stein-identity}
\end{equation}
Therefore
\begin{equation}
  \kappa_S(S)=\E[Q\mid S]
  \tag{5.8}\label{eq:bilinear-scalar-kernel}
\end{equation}
is a scalar Stein kernel for \(S\).  It is nonnegative because
\(\tau\succeq0\).

Conditional on \(Y\), apply \eqref{eq:letwin-qf} to \(X\) with the
symmetric matrix \(B\tau(Y)B^{\mathsf T}\).  By cyclicity of the trace and
\eqref{eq:limit-kernel-energy},
\begin{align}
  \E\Var_X(Q\mid Y)
  &\leq8\E\Tr(A\tau(Y)A\tau(Y))\notag\\
  &\leq16\Tr(A^2).
  \tag{5.9}\label{eq:bilinear-conditional-variance}
\end{align}
Isotropy of \(X\) gives
\[
  \E_X[Q\mid Y]=\Tr(A\tau(Y)),
\]
and \eqref{eq:limit-kernel-trace} gives
\begin{equation}
  \Var\!\left(\E_X[Q\mid Y]\right)
  \leq4\Tr(A^2).
  \tag{5.10}\label{eq:bilinear-trace-variance}
\end{equation}
The law of total variance now yields
\begin{equation}
  \E(Q-\sigma^2)^2\leq20\Tr(A^2).
  \tag{5.11}\label{eq:bilinear-prekernel-variance}
\end{equation}
Since conditional expectation is an \(L^2\) contraction,
\[
  \E(\kappa_S(S)-\sigma^2)^2\leq20\Tr(A^2).
\]
The rescaled kernel \(\kappa_Z(Z)=\kappa_S(S)/\sigma^2\) therefore satisfies
\eqref{eq:bilinear-stein}.  Apply \cref{thm:stein-w2} to obtain
\eqref{eq:bilinear-w2}; scaling \(W_2\) by \(\sigma\) gives
\eqref{eq:bilinear-w2-unnormalized}.
\end{proof}

\begin{corollary}[Jiang--Lee--Vempala conjecture]
\label{thm:gclt}
Assume the estimates in \cref{thm:letwin}.  If \(X,Y\in\R^n\) are
independent isotropic log-concave random vectors and \(G\sim N(0,1)\), then
\begin{equation}
  \boxed{
  W_2^2\!\left(
    \Law\!\left(\frac{\langle X,Y\rangle}{\sqrt n}\right),
    \Law(G)
  \right)
  \leq\frac{20}{n}.}
  \tag{5.12}\label{eq:jlv-normalized}
\end{equation}
Equivalently, if \(G_n\sim N(0,n)\), then
\begin{equation}
  W_2^2\!\left(\Law(\langle X,Y\rangle),\Law(G_n)\right)\leq20.
  \tag{5.13}\label{eq:jlv-unnormalized}
\end{equation}
Thus the dimension-free central limit conjecture of
Jiang, Lee, and Vempala \cite{JLV} follows from the cited estimates.
\end{corollary}

\begin{proof}
Take \(m=n\) and \(B=I_n\) in \cref{thm:bilinear}; then
\(r_4(I_n)=n\) and \(\sigma^2=n\).
\end{proof}

\begin{remark}[Input accounting]
The proof uses a positive-semidefinite Stein kernel for one factor, its full
matrix-energy bound \eqref{eq:limit-kernel-energy}, the trace covariance
bound \eqref{eq:limit-kernel-trace}, the quadratic-form estimate for the
other factor, and the scalar Stein-discrepancy-to-\(W_2\) inequality.
The extension from regular targets uses \cite[Lemma~A.1]{Letwin} but adds no
quantitative constant.  The argument does not use KLS, a density formula for
the bilinear form, or a separate central limit theorem.  The sharp
Hilbert--Schmidt estimate for the full third-moment tensor proved by Chen and
Klartag \cite{ChenKlartag2026} is part of the surrounding theory but is not
used here.  Independence is essential in \eqref{eq:bilinear-means},
\eqref{eq:bilinear-conditional-variance}, and the conditional-mean identity;
isotropy fixes the target variance and that conditional mean.
\end{remark}

\begin{remark}[Spectral participation and the strength of the Stein statement]
Since \(1\leq r_4(B)\leq\operatorname{rank}(B)\), the theorem gives a central limit
regime precisely when the singular-value mass is not concentrated in a few
directions.  The trivial independent coupling gives
\(W_2^2(\Law(Z),N(0,1))\leq2\), so the transport conclusion may be sharpened
to \(\min\{2,20/r_4(B)\}\).  Estimate \eqref{eq:bilinear-stein} contains
more information: it supplies an approximate integration-by-parts identity
for every smooth scalar test.  No claim is made here about total variation,
relative entropy, or a multivariate approximation.
\end{remark}

\section{Models and sharpness checks}
\label{sec:checks}

The following computations are not needed for the proof.  They show why the
spectral participation ratio is the natural parameter, compare the universal constant with
exact model calculations, and verify that the order in
\eqref{eq:bilinear-w2} cannot in general be improved.

\subsection{Rank one and the necessity of spectral participation}

Suppose first that \(X\) and \(Y\) are standard Gaussian and that
\(B=ab^{\mathsf T}\) has rank one.  After normalization,
\[
  \frac{X^{\mathsf T}BY}{\|B\|_{\HS}}
  \stackrel{d}{=}GH,
\]
where \(G,H\) are independent standard Gaussians.  This law is not normal:
it has fourth moment \(9\), rather than \(3\).  Ordinary rank growth does
not by itself repair this obstruction.  Indeed, let \(B_k\) have singular
values \(1,k^{-1},\ldots,k^{-1}\), with \(k-1\) copies of \(k^{-1}\).
Then \(\operatorname{rank}(B_k)=k\), but \(r_4(B_k)\to1\), and the
normalized Gaussian bilinear form converges in \(L^2\) to \(GH\).  Thus the
singular values must become diffuse.  For rank-one \(B\), \(r_4(B)=1\), so
\cref{thm:bilinear} correctly gives only a constant-order estimate.  Sums
of such products also explain the relevance of variance-gamma approximation in low-rank regimes
\cite{GauntVarianceGamma}.

\subsection{Gaussian bilinear forms}

Let \(X\in\R^m\) and \(Y\in\R^n\) be independent standard Gaussian
vectors.  If \((s_i)\) are the nonzero singular values of \(B\), orthogonal
invariance gives
\begin{equation}
  \frac{X^{\mathsf T}BY}{\|B\|_{\HS}}
  \stackrel{d}{=}
  \frac{\sum_i s_iG_iH_i}{(\sum_i s_i^2)^{1/2}},
  \tag{6.1}\label{eq:gaussian-svd}
\end{equation}
with independent standard Gaussian families \((G_i)\) and \((H_i)\).
Since the Stein kernel of \(Y\) is \(I_n\), the pre-kernel in
\eqref{eq:bilinear-prekernel} is \(Q=X^{\mathsf T}BB^{\mathsf T}X\), and
\[
  \Var Q=2\Tr((B^{\mathsf T}B)^2).
\]
The same projection argument as in \cref{thm:bilinear} therefore gives the
sharper model-specific estimate
\begin{equation}
  \E(\kappa_Z(Z)-1)^2
  \leq\frac{2}{r_4(B)},
  \qquad
  W_2^2(\Law(Z),N(0,1))\leq\frac{2}{r_4(B)}.
  \tag{6.2}\label{eq:gaussian-bilinear-bound}
\end{equation}
The fourth moment is exact:
\begin{equation}
  \E Z^4=3+\frac{6}{r_4(B)}.
  \tag{6.3}\label{eq:gaussian-fourth}
\end{equation}
This confirms both the effective-rank scale and the fact that the universal
constant \(20\) is not expected to be optimal.

For completeness, the Gaussian model also tests the trace identity.  If
\(Y\sim N(0,I_n)\), then \(\tau(Y)=I_n\), and for every symmetric \(A\),
\[
  \Var(Y^{\mathsf T}AY)=2\Tr(A^2),
  \qquad
  T_A=\Tr(A^2),
  \qquad
  U_A=\Tr(A^2).
\]
Thus the right side of~(3.4) is
\[
  2\Tr(A^2)-3\Tr(A^2)+\Tr(A^2)=0,
\]
which is exactly
\(\Var(\Tr(A\tau))=0\).  This confirms both the negative sign and the
coefficient \(3\) in the trace identity.

\subsection{Rectangular products of centered exponentials}

Let the coordinates of \(X\in\R^m\) and \(Y\in\R^n\) be independent
centered unit-rate exponentials, and take the two vectors independent of one
another.  Thus \(X_i=E_i-1\), \(Y_j=F_j-1\), and
\(\tau(Y)=\operatorname{diag}(F_1,\ldots,F_n)\).  Write \(b_j\) for the
columns and \(r_i\) for the rows of \(B\), and set
\[
  c_{\rm col}=\sum_j\|b_j\|^4,
  \qquad
  c_{\rm row}=\sum_i\|r_i\|^4,
  \qquad
  c_4=\sum_{i,j}b_{ij}^4.
\]
For \(A=B^{\mathsf T}B\), a direct calculation from
\(\E X_i^4=\E Y_j^4=9\) gives
\begin{equation}
  \Var Q
  =2\Tr(A^2)+3c_{\rm col}+6c_{\rm row}+6c_4
  \leq17\Tr(A^2).
  \tag{6.4}\label{eq:exponential-rectangular}
\end{equation}
Indeed,
\[
  \E\Tr(B\tau B^{\mathsf T}B\tau B^{\mathsf T})
  =\Tr(A^2)+c_{\rm col},
\]
the non-Gaussian diagonal correction to the conditional quadratic variance
is \(6(c_{\rm row}+c_4)\), and the variance of the conditional mean is
\(c_{\rm col}\).  Each of \(c_{\rm col},c_{\rm row},c_4\) is at most
\(\Tr(A^2)\).  Equality in the final bound of
\eqref{eq:exponential-rectangular} holds when \(B\) is diagonal, so the
universal estimate \(20\Tr(A^2)\) is close to the exact pre-kernel variance
on this extremal log-concave model.

Take, in particular, \(B=I_r\).  For
\[
  Z_r=\frac1{\sqrt r}\sum_{i=1}^r X_iY_i,
\]
independence gives the exact moment identities
\begin{equation}
  \E Z_r^3=\frac4{\sqrt r},
  \qquad
  \E Z_r^4=3+\frac{78}{r}.
  \tag{6.5}\label{eq:exponential-moments}
\end{equation}
For any coupling of \(Z_r\) with \(G\sim N(0,1)\), factor
\(Z_r^3-G^3=(Z_r-G)(Z_r^2+Z_rG+G^2)\) and apply Cauchy--Schwarz.
By \eqref{eq:exponential-moments},
\[
  \E(Z_r^2+Z_rG+G^2)^2
  \leq3\!\left(\E Z_r^4
    +\sqrt{\E Z_r^4\E G^4}+\E G^4\right)
  \leq C
\]
for a universal \(C\), independently of the coupling.  Consequently there is a universal
\(c>0\) such that
\begin{equation}
  W_2^2(\Law(Z_r),N(0,1))\geq\frac{c}{r}.
  \tag{6.6}\label{eq:exponential-lower-bound}
\end{equation}
Thus the order \(1/r_4(B)\) in \cref{thm:bilinear} is sharp in general,
even though its numerical constant is not.

\section{Conclusion and qualifications}
\label{sec:limitations}

Under the quantitative estimates stated in \cref{thm:letwin}, this paper
proves the Jiang--Lee--Vempala dimension-free central limit conjecture and a
rectangular extension.  For \(X^{\mathsf T}BY\), the normalized squared
Stein discrepancy and squared \(2\)-Wasserstein error are at most
\(20/r_4(B)\); the choice \(B=I_n\) gives the conjectured dimension-free
bound for \(\langle X,Y\rangle\).  The argument derives the required trace
covariance bound from an exact moment-map identity and a nonnegative deficit
decomposition, then preserves the covariance, matrix-energy, and weighted
Poincar\'e estimates through a general stability theorem for positive Stein
kernels.

The conclusion does not prove KLS.  The reverse implication in
\cite[Theorem~7]{JLV} is an exponent-level statement with fixed
\(\varepsilon\in(0,1/2)\) and an arbitrary loss \(\delta>0\).  For any
\(\eta>0\), the present \(O(1)\) bound is also \(O(n^\eta)\), so that result
with \(\varepsilon=1/2-\eta\) yields only a subpolynomial conclusion for the
KLS constant: for every \(\alpha>0\), it is \(O_\alpha(n^\alpha)\), not a
dimension-free Poincar\'e bound.  Letwin's explicit polylogarithmic estimate is already
stronger than this consequence; recovering KLS would require an endpoint
conversion controlling general nonlinear observables.

The Gaussian and product-exponential models show that
\(1/r_4(B)\) is the correct general order, while also exhibiting slack in
the numerical constant \(20\).  No optimality claim is made for that
constant.  For nonsmooth targets, \cref{prop:limit} constructs a limiting
positive-semidefinite Stein kernel with the required bounds; it does not
assert a canonical kernel or a classical Hessian formula.  The exact mixed
identity itself remains a statement about the regular moment-map Hessian.

\section*{AI disclosure}

GPT 5.6 Sol assisted with literature triage, algebraic cross-checking,
drafting, and production checks, while Claude Opus 5 provided independent
proof verification and writing feedback.
The author independently verified the mathematical arguments and citations
and takes full responsibility for the content of the manuscript.

\bibliographystyle{plainnat}
\bibliography{log-concave-bilinear-forms}

\end{document}